\documentclass[a4paper,draft,12pt]{amsart}
\usepackage{amssymb}
\usepackage{bbm}
\usepackage[all]{xy}
\normalsize

\newtheoremstyle{bracket}{1ex}{2ex}{\rm}{}{\bfseries}{}{0.8em}{\thmnumber{(#2)}}
\newtheoremstyle{thm}{1ex}{2ex}{\itshape}{}{\bfseries}{}{0.9em}{\thmnumber{(#2)}\thmname{ #1}\thmnote{ (#3)}}
\newtheoremstyle{example}{1ex}{2ex}{\rm}{}{\bfseries}{}{0.8em}{\thmnumber{(#2)}\thmname{ #1}}

\theoremstyle{bracket}
\newtheorem{no}{}

\theoremstyle{thm}
\newtheorem{prop}[no]{Proposition}

\theoremstyle{example}
\newtheorem{exa}[no]{Example}

\DeclareMathOperator{\ke}{Ker}

\newcommand{\grmod}{{\sf GrMod}}
\newcommand{\N}{\mathbbm{N}}
\newcommand{\Z}{\mathbbm{Z}}
\newcommand{\Q}{\mathbbm{Q}}
\newcommand{\dfgl}{\mathrel{\mathop:}=}
\newcommand{\Id}{{\rm Id}}
\newcommand{\ilim}{\varinjlim}
\newcommand{\sq}{\hskip2pt\raisebox{.225ex}{\rule{.8ex}{.8ex}\hskip2pt}}
\newcommand{\hm}[3]{{\rm Hom}_{#1}(#2,#3)}
\newcommand{\cechc}[4]{{}^{#1}H^{#2}(#3,#4)}
\newcommand{\grext}[5]{{}^{#1}{\rm Ext}_{#2}^{#3}(#4,#5)}
\newcommand{\grgam}[2]{{}^{#1}\Gamma_{#2}}
\newcommand{\grhm}[4]{{}^{#1}\hm{#2}{#3}{#4}}
\newcommand{\gridt}[3]{{}^{#1}D^{#2}_{#3}}
\newcommand{\grloc}[3]{{}^{#1}H^{#2}_{#3}}
\newcommand{\loccit}{{\it loc.\,cit.}}
\renewcommand{\aa}{{\bf a}}

\newcommand{\snf}{\renewcommand{\thefootnote}{*}\footnotetext{The author was supported by the Swiss National Science Foundation.}}

\begin{document}

\title{Coarsening of graded local cohomology\protect\snf}
\author{Fred Rohrer}
\address{Universit\"at T\"ubingen, Fachbereich Mathematik, Auf der Morgenstelle 10, 72076 T\"u\-bingen, Germany}
\email{fredrohrer0@gmail.com}

\begin{abstract}
Some criteria for graded local cohomology to commute with coarsening functors are proven, and an example is given where graded local cohomology does not commute with coarsening.
\end{abstract}

\maketitle

%%%%%%%%%%%%%%%%%%%%%%%%%%%%%%%%%%%%%%%%%%%%%%%%%%%%%%%%%%%%

Let $G$ be a commutative group, and let $R$ be a $G$-graded commutative ring. The category $\grmod^G(R)$ of $G$-graded $R$-modules is Abelian, fulfils Grothendieck's axiom AB5, and has a projective generator, hence it has enough projectives and injectives. The $G$-graded Hom bifunctor $\grhm{G}{R}{\bullet}{\sq}$ maps two $G$-graded $R$-modules $M$ and $N$ onto the $G$-graded $R$-module $$\grhm{G}{R}{M}{N}=\bigoplus_{g\in G}\hm{\grmod^G(R)}{M}{N(g)}.$$ This contra-covariant bifunctor is left exact in both arguments, and it turns into an exact functor when a projective or injective $G$-graded $R$-module is plugged into its first or second argument, respectively. Thus, by general nonsense we get the $G$-graded Ext bifunctors, i.e., for every $i\in\Z$ a contra-covariant bifunctor $\grext{G}{R}{i}{\bullet}{\sq}$ such that $(\grext{G}{R}{i}{M}{\bullet})_{i\in\Z}$ and $(\grext{G}{R}{i}{\bullet}{M})_{i\in\Z}$ are the right derived cohomological functors of $\grhm{G}{R}{M}{\bullet}$ and $\grhm{G}{R}{\bullet}{M}$, respectively, for every $G$-graded $R$-module $M$. Furthermore, if $F$ is a projective system in $\grmod^G(R)$ over a right filtering ordered set $J$, then composition yields for every $i\in\Z$ a functor $\ilim_J\grext{G}{R}{i}{F}{\bullet}$, and it follows by general nonsense that $(\ilim_J\grext{G}{R}{i}{F}{\bullet})_{i\in\Z}$ is the right derived cohomological functor of $\ilim_J\grhm{G}{R}{F}{\bullet}$. In particular, given a graded ideal $\mathfrak{a}\subseteq R$ we can consider the projective system $(R/\mathfrak{a}^n)_{n\in\N}$ over $\N$ (where the morphisms are the canonical ones), and then the above construction yields the $G$-graded local cohomology functors $\grloc{G}{i}{\mathfrak{a}}$, i.e., the right derived cohomological functor of the $G$-graded $\mathfrak{a}$-torsion functor $\grgam{G}{\mathfrak{a}}$.

The above evokes the natural question: Do we get the same when we first apply graded local cohomology and then forget the graduation on the resulting module, and when we first forget the graduation on a given module and then apply ungraded local cohomology? The main reason for this to be nontrivial is that the functor that forgets the graduation does not necessarily preserve injectivity of objects -- see \cite[A.I.2.6]{no1} for a counterexample.

In case $G=\Z$ and $R$ is Noetherian, the above question is extensively discussed and positively answered in \cite[Chapter 12]{bs}\footnote{In fact, in \loccit\,\! it is supposed that the ungraded ring underlying $R$ is Noetherian, but in this special situation this is by \cite[A.II.3.5]{no1} equivalent to $R$ being Noetherian (as a graded ring, i.e., every increasing sequence of graded ideals is stationary).}. Here we study this question in case $G$ and $R$ are arbitrary, and we present some criteria for a positive answer. The importance of this stems mainly from toric geometry, where arbitrary groups of degrees (of finite type) are ubiquitous (see \cite{cox}). Since it seems also useful and enlightening, we consider a more general situation. Namely, we replace the functor that forgets the graduation by coarsening functors, of which we remind the reader now.

Let $\psi\colon G\twoheadrightarrow H$ be an epimorphism of groups. We define an $H$-graded ring $R_{[\psi]}=\bigoplus_{h\in H}(\bigoplus_{g\in\psi^{-1}(h)}R_g)$, and analogously, we get a functor $$\bullet_{[\psi]}\colon\grmod^G(R)\rightarrow\grmod^H(R_{[\psi]}),$$ called the $\psi$-coarsening functor. This functor is faithful, conservative, exact, has a right adjoint, and thus commutes with inductive limits and with finite projective limits. Our question is now whether graded local cohomology commutes with coarsening, i.e., whether the diagram of categories $$\xymatrix@C50pt@R20pt{\grmod^G(R)\ar[r]^{\grloc{G}{i}{\mathfrak{a}}}\ar[d]_{\bullet_{[\psi]}}&\grmod^G(R)\ar[d]^{\bullet_{[\psi]}}\\\grmod^H(R_{[\psi]})\ar[r]^{\grloc{H}{i}{\mathfrak{a}_{[\psi]}}}&\grmod^H(R_{[\psi]})}$$ quasicommutes.

Our first approach uses the above definition of graded local cohomology and hence comes down to the analogous question for graded Ext functors. Known facts about the analogous question for graded Hom functors (see \cite{gpn}, \cite{gpmn}) together with $\delta$-functor techniques will allow us to quickly get our first two criteria (\ref{40}) and (\ref{50}). A second approach uses the relation between local cohomology and \v{C}ech cohomology and will result in a third criterion (\ref{60}).

\smallskip

We start by defining some canonical morphisms of functors.

\begin{no}\label{10}
For $G$-graded $R$-modules $M$ and $N$ and $g\in G$, we have a mono\-morphism of groups $$\hm{\grmod^G(R)}{M}{N(g)}\rightarrowtail\hm{\grmod^H(R_{[\psi]})}{M_{[\psi]}}{N_{[\psi]}(\psi(g))}$$ with $u\mapsto u_{[\psi]}$, inducing a monomorphism of $H$-graded $R_{[\psi]}$-modules $$h_{\psi}(M,N)\colon\grhm{G}{R}{M}{N}_{[\psi]}\rightarrowtail\grhm{H}{R_{[\psi]}}{M_{[\psi]}}{N_{[\psi]}}.$$ Varying $M$ and $N$, we get a monomorphism of bifunctors $$h_{\psi}\colon\grhm{G}{R}{\bullet}{\sq}_{[\psi]}\rightarrowtail\grhm{H}{R_{[\psi]}}{\bullet_{[\psi]}}{\sq_{[\psi]}}.$$

If $F$ is a projective system in $\grmod^G(R)$ over a right filtering ordered set $J$, then $h_{\psi}$ induces a monomorphism of functors $$h_{F,\psi}\dfgl\ilim_Jh_{\psi}(F,\bullet)\colon\ilim_J\grhm{G}{R}{F}{\bullet}_{[\psi]}\rightarrowtail\ilim_J\grhm{H}{R_{[\psi]}}{F_{[\psi]}}{\bullet_{[\psi]}}.$$
\end{no}

One may note that the above monomorphisms are not necessarily isomorphisms. An easy example of this phenomenon is obtained by considering $G$ infinite, $H=0$ and $R\neq 0$ trivially $G$-graded (i.e., $R_{[0]}=R_0$), for then $h_0(\bigoplus_{g\in G}R(g),R)$ coincides with the canonical injection $R_0^{\oplus G}\rightarrowtail R_0^G$. In fact, it follows from \cite[A.I.2.11]{no1} and \cite[3.4]{gpmn} that if $M$ is a $G$-graded $R$-module, then $h_0(M,\bullet)$ is an isomorphism if and only if $G$ is finite or $M$ is small\footnote{A $G$-graded $R$-module $M$ is called small if $\hm{\grmod^G(R)}{M}{\bullet}$ commutes with direct sums.}.

\begin{no}\label{20}
If $M$ is a $G$-graded $R$-module, then exactness of $\bullet_{[\psi]}$ yields exact $\delta$-functors $$(\grext{G}{R}{i}{M}{\bullet}_{[\psi]})_{i\in\Z},\quad(\grext{H}{R_{[\psi]}}{i}{M_{[\psi]}}{\bullet_{[\psi]}})_{i\in\Z},$$$$(\grext{G}{R}{i}{\bullet}{M}_{[\psi]})_{i\in\Z},\quad(\grext{H}{R_{[\psi]}}{i}{\bullet_{[\psi]}}{M_{[\psi]}})_{i\in\Z},$$ and implies effaceability of $\grext{G}{R}{i}{M}{\bullet}_{[\psi]}$ and $\grext{G}{R}{i}{\bullet}{M}_{[\psi]}$ for $i>0$. But $\grext{H}{R_{[\psi]}}{i}{\bullet_{[\psi]}}{M_{[\psi]}}$ is effaceable for $i>0$, too, for $\bullet_{[\psi]}$ preserves projectivity (\cite[A.I.2.2]{no1}). Thus, $$(\grext{G}{R}{i}{M}{\bullet}_{[\psi]})_{i\in\Z},\quad(\grext{G}{R}{i}{\bullet}{M}_{[\psi]})_{i\in\Z},\quad(\grext{H}{R_{[\psi]}}{i}{\bullet_{[\psi]}}{M_{[\psi]}})_{i\in\Z}$$ are the right derived cohomological functors of $\grhm{G}{R}{M}{\bullet}_{[\psi]}$,\linebreak $\grhm{G}{R}{\bullet}{M}_{[\psi]}$, and $\grhm{H}{R_{[\psi]}}{\bullet_{[\psi]}}{M_{[\psi]}}$, respectively. So, there are unique morphisms of $\delta$-functors $$(h_{\psi}^i(M,\bullet))_{i\in\Z}\colon(\grext{G}{R}{i}{M}{\bullet}_{[\psi]})_{i\in\Z}\rightarrow (\grext{H}{R_{[\psi]}}{i}{M_{[\psi]}}{\bullet_{[\psi]}})_{i\in\Z}$$ and $$(h_{\psi}^i(\bullet,M))_{i\in\Z}\colon(\grext{G}{R}{i}{\bullet}{M}_{[\psi]})_{i\in\Z}\rightarrow(\grext{H}{R_{[\psi]}}{i}{\bullet_{[\psi]}}{M_{[\psi]}})_{i\in\Z}$$ with $h_{\psi}^0(M,\bullet)=h_{\psi}(M,\bullet)$ and $h_{\psi}^0(\bullet,M)=h_{\psi}(\bullet,M)$, and they define for every $i\in\Z$ a morphism of bifunctors $$h_{\psi}^i\colon\grext{G}{R}{i}{\bullet}{\sq}_{[\psi]}\rightarrow\grext{H}{R_{[\psi]}}{i}{\bullet_{[\psi]}}{\sq_{[\psi]}}$$ with $h_{\psi}^0=h_{\psi}$ (\cite[2.3]{t}). Furthermore, if $h_{\psi}(\bullet,M)$ is an isomorphism then so is $(h_{\psi}^i(\bullet,M))_{i\in\Z}$, and thus if $h_{\psi}$ is an isomorphism then so is $h_{\psi}^i$ for every $i\in\Z$.

If $F$ is a projective system in $\grmod^G(R)$ over a right filtering ordered set $J$, then exactness of $\bullet_{[\psi]}$ yields exact $\delta$-functors $$(\ilim_J\grext{G}{R}{i}{F}{\bullet}_{[\psi]})_{i\in\Z},\quad(\ilim_J\grext{H}{R_{[\psi]}}{i}{F_{[\psi]}}{\bullet_{[\psi]}})_{i\in\Z},$$ of which the first is universal, hence a unique morphism of $\delta$-functors $$(h_{F,\psi}^i)_{i\in\Z}\colon(\ilim_J\grext{G}{R}{i}{F}{\bullet}_{[\psi]})_{i\in\Z}\rightarrow(\ilim_J\grext{H}{R_{[\psi]}}{i}{F_{[\psi]}}{\bullet_{[\psi]}})_{i\in\Z}$$ with $h_{F,\psi}^0=h_{F,\psi}$.
\end{no}

\begin{no}\label{30}
Let $\mathfrak{a}\subseteq R$ be a graded ideal. We consider the projective system $R/\mathfrak{A}=(R/\mathfrak{a}^n)_{n\in\N}$ in $\grmod^G(R)$ over $\N$. There is by (\ref{20}) a canonical morphism of $\delta$-functors $$(h_{R/\mathfrak{A},\psi}^i)_{i\in\Z}\colon(\grloc{G}{i}{\mathfrak{a}}(\bullet)_{[\psi]})_{i\in\Z}\rightarrow(\grloc{H}{i}{\mathfrak{a}_{[\psi]}}(\bullet_{[\psi]}))_{i\in\Z},$$ and $h_{R/\mathfrak{A},\psi}\colon\grgam{G}{\mathfrak{a}}(\bullet)_{[\psi]}\rightarrowtail\grgam{H}{\mathfrak{a}_{[\psi]}}(\bullet_{[\psi]})$ is the identity morphism.
\end{no}

Related to graded local cohomology are graded higher ideal transformations: The $\delta$-functor $(\gridt{G}{i}{\mathfrak{a}}(\bullet))_{i\in\Z}\dfgl(\ilim_{n\in\N}\grext{G}{R}{i}{\mathfrak{a}^n}{\bullet})_{i\in\Z}$ is the right derived cohomological functor of $\ilim_{n\in\N}\grhm{G}{R}{\mathfrak{a}^n}{\bullet}$. It plays an important role in the relation between local cohomology and sheaf cohomology, on projective schemes in case $G=\Z$ (\cite[20.4.4]{bs}), and on toric schemes in case of more general groups of degree. Setting $\mathfrak{A}=(\mathfrak{a}^n)_{n\in\N}$ we get from (\ref{20}) a morphism of $\delta$-functors $$(h_{\mathfrak{A},\psi}^i)_{i\in\Z}\colon(\gridt{G}{i}{\mathfrak{a}}(\bullet)_{[\psi]})_{i\in\Z}\rightarrow(\gridt{H}{i}{\mathfrak{a}_{[\psi]}}(\bullet_{[\psi]}))_{i\in\Z}.$$

We show now that local cohomology commutes with coarsening if and only if higher ideal transformation does so.

\begin{prop}\label{70}
The morphism of $\delta$-functors $$(h_{\mathfrak{A},\psi}^i)_{i\in\Z}\colon(\gridt{G}{i}{\mathfrak{a}}(\bullet)_{[\psi]})_{i\in\Z}\rightarrow(\gridt{H}{i}{\mathfrak{a}_{[\psi]}}(\bullet_{[\psi]}))_{i\in\Z}$$ is an isomorphism if and only if the morphism of $\delta$-functors $$(h_{R/\mathfrak{A},\psi}^i)_{i\in\Z}\colon(\grloc{G}{i}{\mathfrak{a}}(\bullet)_{[\psi]})_{i\in\Z}\rightarrow(\grloc{H}{i}{\mathfrak{a}_{[\psi]}}(\bullet_{[\psi]}))_{i\in\Z}$$ is so.
\end{prop}

\begin{proof}
Analogously to \cite[2.2.4]{bs} we get an exact sequence of functors $$0\rightarrow\grgam{G}{\mathfrak{a}}\rightarrow\Id_{\grmod^G(R)}\rightarrow\gridt{G}{0}{\mathfrak{a}}\rightarrow\grloc{G}{1}{\mathfrak{a}}\rightarrow 0$$ and a morphism of $\delta$-functors $(\zeta^i_{\mathfrak{a}})_{i\in\Z}\colon(\gridt{G}{i}{\mathfrak{a}})_{i\in\Z}\rightarrow(\grloc{G}{i+1}{\mathfrak{a}})_{i\in\Z}$ such that $\zeta_{\mathfrak{a}}^i$ is an isomorphism for $i>0$. Now the Five Lemma yields the claim.
\end{proof}

\begin{prop}\label{40}
If\/ $\ke(\psi)$ is finite then $$(h_{R/\mathfrak{A},\psi}^i)_{i\in\Z}\colon(\grloc{G}{i}{\mathfrak{a}}(\bullet)_{[\psi]})_{i\in\Z}\rightarrow(\grloc{H}{i}{\mathfrak{a}_{[\psi]}}(\bullet_{[\psi]}))_{i\in\Z}$$ is an isomorphism.
\end{prop}

\begin{proof}
By (\ref{20}) it suffices to show that $$h_{\psi}\colon\grhm{G}{R}{\bullet}{\sq}_{[\psi]}\rightarrowtail\grhm{H}{R_{[\psi]}}{\bullet_{[\psi]}}{\sq_{[\psi]}}$$ is an epimorphism. Let $M$ and $N$ be $G$-graded $R$-modules, and let $f\in\grhm{H}{R_{[\psi]}}{M_{[\psi]}}{N_{[\psi]}}$. By \cite[A.I.2.10]{no1}, there is a finite subset $S\subseteq H$ with $f((M_{[\psi]})_h)\subseteq\sum_{l\in S}(N_{[\psi]})_{h+l}$ for $h\in H$. Finiteness of $\ke(\psi)$ implies that $T\dfgl\psi^{-1}(S)\subseteq G$ is finite. If $g\in G$ then $$\textstyle f(M_g)\subseteq f((M_{[\psi]})_{\psi(g)})\subseteq\sum_{l\in S}(N_{[\psi]})_{\psi(g)+l}=$$$$\textstyle\sum_{l\in T}(N_{[\psi]})_{\psi(g)+\psi(l)}=\sum_{l\in T}\sum_{k\in\psi^{-1}(\psi(g+l))}N_k=\sum_{k\in T}N_{g+k},$$ thus $f\in\grhm{G}{R}{M}{N}$, again by \cite[A.I.2.10]{no1}.
\end{proof}

The hypothesis of (\ref{40}) is fulfilled if $G$ is of finite type and $\psi$ is the projection onto $G$ modulo its torsion subgroup. In this sense, local cohomology does not care about torsion in the group of degrees.

\begin{prop}\label{50}
If\/ $\mathfrak{a}^n$ has a projective resolution with components of finite type for every $n\in\N$, then $$(h_{R/\mathfrak{A},\psi}^i)_{i\in\Z}\colon(\grloc{G}{i}{\mathfrak{a}}(\bullet)_{[\psi]})_{i\in\Z}\rightarrow(\grloc{H}{i}{\mathfrak{a}_{[\psi]}}(\bullet_{[\psi]}))_{i\in\Z}$$ is an isomorphism.
\end{prop}

\begin{proof}
It suffices to show that $h_{\psi}^i(M,\bullet)$ is an isomorphism for every $i\in\Z$ and every $G$-graded $R$-module $M$ that has a projective resolution $P$ with components of finite type. As $h_{\psi}^i(M,\bullet)=H^i(h_{\psi}(P,\bullet))$ for $i\in\Z$, it suffices to show that $h_{\psi}(M,N)$ is an isomorphism for every $G$-graded $R$-module $M$ of finite type and every $G$-graded $R$-module $N$. In this situation, $M_{[\psi]}$ is of finite type, and hence $$h_{0}(M,N)\colon\grhm{G}{R}{M}{N}_{[0]}\rightarrowtail\grhm{0}{R_{[0]}}{M_{[0]}}{N_{[0]}}$$ and $$h_{0}(M_{[\psi]},N_{[\psi]})\colon\grhm{H}{R_{[\psi]}}{M_{[\psi]}}{N_{[\psi]}}_{[0]}\rightarrowtail\grhm{0}{R_{[0]}}{M_{[0]}}{N_{[0]}}$$ are isomorphisms (\cite[A.I.2.11]{no1}). It follows that $$(h_{\psi}(M,N))_{[0]}\colon(\grhm{G}{R}{M}{N}_{[\psi]})_{[0]}\rightarrowtail(\grhm{H}{R_{[\psi]}}{M_{[\psi]}}{N_{[\psi]}})_{[0]}$$ is an isomorphism, too, and since coarsening functors are conservative, this implies that $h_{\psi}(M,N)$ is an isomorphism as desired.\footnote{Using \cite[3.1; 3.4]{gpmn} one sees that this proof also applies if every power of $\mathfrak{a}$ has a projective resolution with small components. But since a projective $G$-graded $R$-module is small if and only if it is of finite type (\cite[II.1.2]{bass}), this yields no improvement.}
\end{proof}

The hypothesis of (\ref{50}) is fulfilled if $R$ is coherent and $\mathfrak{a}$ is of finite type, and thus in particular if $R$ is Noetherian.

\smallskip

Our third criterion makes use of the so-called ITI property. The $G$-graded ring $R$ is said to have ITI with respect to $\mathfrak{a}$ if the graded $\mathfrak{a}$-torsion submodule $\grgam{G}{\mathfrak{a}}(I)$ of an injective $G$-graded $R$-module $I$ is injective. This property, strictly weaker than graded Noetherianness, lies at the heart of many basic properties of local cohomology and thus is a very natural hypothesis. For more details and examples about ITI, we refer the reader to \cite{qr}.

\begin{prop}\label{60}
If\/ $\mathfrak{a}$ has a finite set $E$ of homogeneous generators such that $R$ and $R_{[\psi]}$ have ITI with respect to $\langle a\rangle$ for every $a\in E$, then $$(h_{R/\mathfrak{A},\psi}^i)_{i\in\Z}\colon(\grloc{G}{i}{\mathfrak{a}}(\bullet)_{[\psi]})_{i\in\Z}\rightarrow(\grloc{H}{i}{\mathfrak{a}_{[\psi]}}(\bullet_{[\psi]}))_{i\in\Z}$$ is an isomorphism.
\end{prop}

\begin{proof}
We choose a counting $\aa=(a_i)_{i=1}^n$ of $E$ and write $(\cechc{G}{i}{\aa}{\bullet})_{i\in\Z}$ for the $G$-graded \v{C}ech cohomology with respect to $\aa$. This exact $\delta$-functor is obtained by taking homology of the $G$-graded \v{C}ech cocomplex, and since that cocomplex is obtained by taking direct sums of modules of fractions with homogeneous denominators, it commutes with coarsening. As $\grloc{G}{0}{\mathfrak{a}}(\bullet)=\cechc{G}{0}{\aa}{\bullet}$ there is a unique morphism of $\delta$-functors $$(g_{\aa}^i)_{i\in\Z}\colon(\grloc{G}{i}{\mathfrak{a}}(\bullet))_{i\in\Z}\rightarrow(\cechc{G}{i}{\aa}{\bullet})_{i\in\Z}$$ with $g^0_{\aa}=\Id_{\grloc{G}{0}{\mathfrak{a}}}$. So, it suffices to show that $(g_{\aa}^i)_{i\in\Z}$ is an isomorphism if $R$ has ITI with respect to $\langle a_i\rangle$ for every $i\in[1,n]$. For this it suffices to show that $\cechc{G}{i}{\aa}{\bullet}$ is effaceable for $i>0$. But this can be shown analogously to the ungraded case (\cite[5.1.19]{bs})\footnote{In \loccit\,\! one makes use only of the ITI property, but not of Noetherianness.}.
\end{proof}

In \cite{sch}, Schenzel characterised ideals $\mathfrak{a}$ such that local cohomology with respect to $\mathfrak{a}$ is isomorphic to \v{C}ech cohomology with respect to a generating system of $\mathfrak{a}$. The above proof together with a graded analogue of Schenzel's result yields a further coarsening criterion.

\smallskip

We end this note with an example showing that graded local cohomology does not necessarily commute with coarsening.

\begin{exa}
Let $K$ be a field, let $R$ denote the algebra of the additive monoid $(\Q_{\geq 0},+)$ of positive rational numbers over $K$, furnished with its canonical $\Q$-graduation, and let $\{e_{\alpha}\mid\alpha\in\Q_{\geq 0}\}$ denote its canonical basis (as a $K$-vector space). We consider coarsening with respect to the zero morphism $0\colon\Q\rightarrow 0$. The graded ideal $\mathfrak{m}\dfgl\langle e_{\alpha}\mid\alpha>0\rangle_R$ of $R$ is idempotent, of countable type, but not of finite type, and thus it is not small by \cite[2$^\circ$]{ren}. Hence, by \cite[3.4]{gpmn} there exists a $\Q$-graded $R$-module $N$ such that the canonical monomorphism $$\grhm{\Q}{R}{\mathfrak{m}}{N}_{[0]}\rightarrowtail\grhm{0}{R_{[0]}}{\mathfrak{m}_{[0]}}{N_{[0]}}$$ is not an epimorphism. By idempotency of $\mathfrak{m}$ it equals the canonical morphism $\gridt{\Q}{0}{\mathfrak{m}}(N)_{[0]}\rightarrow\gridt{0}{0}{\mathfrak{m}_{[0]}}{N_{[0]}}$, and so the exact sequence $$0\rightarrow\grgam{\Q}{\mathfrak{m}}\rightarrow\Id_{\grmod^{\Q}(R)}\rightarrow\gridt{\Q}{0}{\mathfrak{m}}\rightarrow\grloc{\Q}{1}{\mathfrak{m}}\rightarrow 0$$ together with the Snake Lemma shows that the canonical morphism $\grloc{\Q}{1}{\mathfrak{m}}(N)_{[0]}\rightarrow\grloc{0}{1}{\mathfrak{m}_{[0]}}(N_{[0]})$ is a monomorphism, but not an epimorphism.
\end{exa}

\smallskip

\noindent\textbf{Acknowledgement.} I am grateful to Markus Brodmann for his encouraging support during the writing of this article. I also thank the referee for his careful reading and the suggested improvements.

%%%%%%%%%%%%%%%%%%%%%%%%%%%%%%%%%%%%%%%%%%%%%%%%%%%%%%%%%%%%


\begin{thebibliography}{99}
\bibitem{bass} H. Bass. \emph{Algebraic K-theory.} W. A. Benjamin, New York, 1968.
\bibitem{bs} M. P. Brodmann, R. Y. Sharp. \emph{Local cohomology: an algebraic introduction with geometric applications.} Cambridge Stud. Adv. Math. 60, Cambridge University Press, Cambridge, 1998.
\bibitem{cox} D. A. Cox. \emph{The homogeneous coordinate ring of a toric variety.} J. Algebraic Geom. 4 (1995), 17--50.
\bibitem{gpn} J. L. G\'omez Pardo, C. N\v{a}st\v{a}sescu. \emph{Topological aspects of graded rings.} Comm. Algebra 21 (1993), 4481--4493.
\bibitem{gpmn} J. L. G\'omez Pardo, G. Militaru, C. N\v{a}st\v{a}sescu. \emph{When is ${\rm HOM}_R(M,-)$ equal to ${\rm Hom}_R(M,-)$ in the category $R$-$gr$?} Comm. Algebra 22 (1994), 3171--3181.
\bibitem{t} A. Grothendieck. \emph{Sur quelques points d'alg\`ebre homologique.} Tohoku Math. J. (2) 9 (1957), 119--221.
\bibitem{no1} C. N\v{a}st\v{a}sescu, F. van Oystaeyen. \emph{Graded ring theory.} North-Holland Math. Library 28, North-Holland Publishing Co., Amsterdam, 1982.
\bibitem{qr} P. H. Quy, F. Rohrer. \emph{Bad behaviour of injective modules.} Preprint (2011).
\bibitem{ren} R. Rentschler. \emph{Sur les modules $M$ tels que ${\rm Hom}(M,-)$ commute avec les sommes directes.} C. R. Acad. Sc. Paris 268 (1969), 930--933.
\bibitem{sch} P. Schenzel. \emph{Proregular sequences, local cohomology, and completion.} Math. Scand. 92 (2003), 161--180.
\end{thebibliography}
\end{document}